\documentclass[12pt]{amsart} 
 \usepackage{amssymb, latexsym, amsmath}

 \usepackage{amssymb, latexsym, amsmath} 
 \begin{document} 
 \theoremstyle{plain} 
 \newtheorem{theorem}{Theorem}[section] 
 \newtheorem{lemma}[theorem]{Lemma} 
 \newtheorem{corollary}[theorem]{Corollary} 
 \newtheorem{proposition}[theorem]{Proposition}

\theoremstyle{definition} 
\newtheorem*{definition}{Definition}
\newtheorem{example}[theorem]{Example}
\newtheorem{remark}[theorem]{Remark}
\title[Abelian Hopf Galois structures]{ Abelian Hopf Galois structures from almost trivial commutative nilpotent algebras}
\author{Lindsay N. Childs}
\address{Department of Mathematics and Statistics\\
University at Albany\\
Albany, NY 12222}
\email{lchilds@albany.edu}
\date{\today}
\newcommand{\Zpn}{\mathbb{Z}/p^n\mathbb{Z}}
\newcommand{\NN}{\mathbb{N}} 
\newcommand{\QQ}{\mathbb{Q}} 
\newcommand{\RR}{\mathbb{R}} 
\newcommand{\CC}{\mathbb{C}} 
\newcommand{\FF}{\mathbb{F}} 
\newcommand{\ZZ}{\mathbb{Z}}
\newcommand{\ZZm}{\mathbb{Z}/m\mathbb{Z}}
\newcommand{\ZZp}{\mathbb{Z}/p\mathbb{Z}}
\newcommand{\T}{\Theta} 
\newcommand{\af}{\alpha} 
\newcommand{\bt}{\beta} 
\newcommand{\ep}{\epsilon} 
\newcommand{\ee}{\end{eqnarray}} 
\newcommand{\no}{\noindent} 

\newcommand{\ben}{\begin{eqnarray*}} 
\newcommand{\een}{\end{eqnarray*}} 
\newcommand{\dis}{\displaystyle} 
\newcommand{\beal}{\[ \begin{aligned}} 
\newcommand{\eeal}{ \end{aligned} \]} 
\newcommand{\sg}{\sigma} 
\newcommand{\eps}{\varepsilon} 
\newcommand{\lb}{\lambda} 
\newcommand{\gm}{\gamma} 
\newcommand{\dt}{\delta} 
\newcommand{\zt}{\zeta} 
\newcommand{\Gm}{\Gamma} 
\newcommand{\bpm}{\begin{pmatrix}} 
\newcommand{\epm}{\end{pmatrix}} 
\newcommand{\Fp}{\mathbb{F}_p} 
\newcommand{\Fpx}{\mathbb{F}_p^{\times}} 

\newcommand{\mcF}{\mathcal{F}} 
\newcommand{\mcp}{\mathcal{p}}
\newcommand{\mbF} {mathbf{F}}
\newcommand{\mcG}{\mathcal{G}}
\newcommand{\mcD}{\mathcal{D}}
\newcommand{\mcP}{\mathcal{P}}
\newcommand{\mcQ}{\mathcal{Q}}
\newcommand{\mcU}{\mathcal{U}}
\newcommand{\mcS}{\mathcal{S}}

\newcommand{\mcH}{\mathcal{H}} 
\newcommand{\hb}{\hat{b}} 
\newcommand{\ha}{\hat{a}}
\newcommand{\hc}{\hat{c}}
\newcommand{\z}{\hat{\zeta}} 
\newcommand{\noi}{\noindent} 
\newcommand{\hv}{\hat{v}} 
\newcommand{\hp}{\hat{p}} 
\newcommand{\hpi}{\hat{p}_{\psi_i}}
\newcommand{\hpw}{\hat{p}_{\psi_1}}
\newcommand{\hpe}{\hat{p}_{\psi_e}}
\newcommand{\hpd}{\hat{p}_{\psi_d}}
\newcommand{\NBP}{Norm_B(\mcP)}
\newcommand{\mcR}{\mathcal{R}}
\newcommand{\hpx}{\hat{p}_{\chi}}
\newcommand{\Af}{\text{Aff} }
\newcommand{\tb}{\textbullet \ }

\newcommand{\mcO}{\mathcal{O}}
\newcommand{\GL}{\mathrm{GL}}
\newcommand{\mbx}{\mathbf{x}}
\newcommand{\mby}{\mathbf{y}}
\newcommand{\mfOK}{\mathfrak{O}_K}
\newcommand{\mfOL}{\mathfrak{O}_L}

\newcommand{\mfA}{\mathfrak{A}}

\newcommand{\M}{\mathrm{M}}
\newcommand{\Aut}{\mathrm{Aut}}
\newcommand{\End}{\mathrm{End}}
\newcommand{\Perm}{\mathrm{Perm}}
\newcommand{\Hol}{\mathrm{Hol}}
\newcommand{\Sta}{\mathrm{Sta}}
\newcommand{\GO}{\mathrm{GO}}
\newcommand{\PGL}{\mathrm{PGL}}
\newcommand{\diag}{\mathrm{diag}}

\begin{abstract} Let $L/K$ be a Galois extension of fields with Galois group $G \cong (\Fp^n, +)$, an elementary abelian $p$-group of rank $n$ for $p$ an odd prime.  It is known that nilpotent $\Fp$-algebra structures $A$ on $G$ yield regular subgroups of the holomorph $\Hol(G)$, hence Hopf Galois structures on $L/K$.  In this paper we illustrate the richness of Hopf Galois structures on $L/K$ by examining the case where $A$ is abelian of $\Fp$-dimension $n$ where $\dim(A^2) = 1$.   We determine the number of Hopf Galois structures that arise in these cases,  describe those structures explicitly, and estimate the extent of failure of surjectivity of the Galois correspondence for those structures. 
\end{abstract}
\maketitle

\section{Introduction}  
 In 1969, Chase and Sweedler [CS69] defined the notion of a Hopf Galois extension of fields by abstracting the formal properties of a classical Galois extension of fields.  In 1987, Greither and Pareigis [GP87] discovered that a classical Galois extension $L/K$ of fields with Galois group $G$ could also be a Hopf Galois extension for a $K$-Hopf algebra $H$ other than $H = KG$, the group ring of the Galois group, acting in the obvious way on $L$.  They showed that determining the number of Hopf Galois structures on $L/K$ depends solely on the Galois group $G$. More precisely, the Hopf Galois structures correspond to regular subgroups $N$ of the permutation group of $G$ that are normalized by the image $\lb(G)$ of the left regular representation of $G$ in $\Perm(G)$.   In many examples the subgroup $N$ of $\Perm(G)$ need not be isomorphic to $G$, so we define  the \emph{type} of the Hopf Galois extension of $L/K$ corresponding to $N$ to be the isomorphism class of the group $N$.  
 
Since the appearance of [GP87] there has been a fairly steady sequence of papers studying the number of Hopf Galois structures on a Galois extension of fields $L/K$ with Galois group $G$.  These range from Byott's uniqueness paper [By96] and his  theorem [By04] that if $G$ is a non-abelian simple group then $L/K$ has exactly two Hopf Galois structures, to papers that for suitable Galois groups $G$ describe large numbers of Hopf Galois structures on $L/K$ , e.g. [Ch15], or describe Hopf Galois structures of all possible types, e. g. [AB18]. 

Counting Hopf Galois structures on a field extension with Galois group $G$ is often made easier by translating the problem of finding regular subgroups  of $\Perm(G)$ that are isomorphic to a given group $N$ and are normalized by $\lb(G)$ to a problem of finding regular subgroups  of $\Hol(N)$ that are isomorphic to $G$.  This translation from $\Perm(G)$ to $\Hol(N)$ was first codified in [By96], and has been the approach of choice for most papers devoted to counting Hopf Galois structures.  

 In [CDVS06] and subsequently in [FCC12] Caranti, et. al.  showed that for a finite abelian $p$-group, any commutative regular subgroup of $\Hol(G)$ can be obtained as the circle, or adjoint, group of a commutative nilpotent algebra structure $(G, +, \cdot)$ on the additive group $G$.  Meanwhile, Rump [Ru07] defined a left brace and showed that if $(A, +, \cdot)$ is a radical ring, then $(A, \circ, +)$ is a left brace, and Guarneri and Vendramin [GV17] extended the concept to that of a skew left brace by relaxing the commutativity assumption on the operation $+$.  In particular, they characterized skew braces as follows:

\begin{theorem}  Let $(B, \circ, \star)$ be a set with two group operations $\circ$ and $\star$.  Let $\lambda_{\circ}, \lambda_{\star}:  B  \to \Perm(B)$ be the two left regular representation maps, defined by $\lambda_{\circ}(b) (x) = b \circ x$, $\lambda_{\star}(b) (x) = b \star x$.  Let $\Hol(B, \star) \subset \Perm(B)$ be the normalizer of $\lambda_{\star}(B)$ in $\Perm(B)$.  Then $B$ is a skew left brace if and only if $\lambda_{\circ}(B) \subset \Hol(B, \star)$.  \end{theorem}

Subsequently, Byott and Vendramin [SV18] showed that  $(B, \circ, \star)$ is a skew left brace if and only if  there exists a Galois extension $L/K$ with Galois group $G$ and a Hopf Galois structure of type $N$ so that $G \cong (B, \circ)$ and $N \cong (B, \star)$.  Their observation follows from the fact from Greither-Pareigis that if $N = (B, \star)
$ is normalized by $\lambda_{\circ}(B) = \lambda(G)$ in $\Perm(G)$, then $N$ corresponds to a Hopf Galois structure on $L/K$, and conversely, if $L/K$ is $G$-Galois and has a Hopf Galois structure corresponding to the regular subgroup $N$ of $\Perm(G)$ normalized by  $\lambda_{\circ}(G)$ in $\Perm(G)$, then $N$ defines a new group structure $(G, \star) \cong N$ on $G$ which makes $(G, \circ, \star)$ into a skew brace.

To illustrate the richness of Hopf Galois structures, especially on Galois extensions with Galois group an elementary abelian $p$-group, insights can be gained by just looking at those arising from nilpotent $\Fp$-algebras.  
In this paper we look at Hopf Galois structures corresponding to a class of commutative nilpotent $\Fp$-algebras $A$ of $\Fp$-dimension $n$ that are ``almost'' trivial.  Thus we assume that $A$ has the property that $\dim_{\Fp}(A^2) = 1$  and $A^3 = 0$.  If $A^2 = 0$, then the Hopf Galois structure on a Galois extension with Galois group $G \cong (A, +)$ is unique, namely that given by the Galois group, so our examples of nilpotent algebras are about as close to being trivial as possible.

We determine the isomorphism types of commutative nilpotent $\Fp$-algebras $A$ of dimension $n$ with $A^3 = 0$ and $\dim(A^2) = 1$, and determine the number of regular subgroups of $\Hol(G)$ associated to each isomorphism type.   For $n = 4$  this approach yields more than $p^9$ regular subgroups.  We describe the Hopf Galois structure on $L/K$ corresponding to each regular subgroup arising from a given isomorphism type of algebra.  We also explicitly describe the Hopf algebra action on $L/K$, and estimate the extent of failure of the Galois correspondence for the Hopf Galois structure to map onto the intermediate fields between $K$ and $L$.

Throughout, let $L/K$ be a Galois extension of fields with Galois group $\Gm$, an elementary abelian $p$-group of order $p^n$, $p$ an odd prime.  For a finite abelian $p$-group $G$ with operation $+$, a ring structure $A = (G, +, \cdot)$ on $G$ will be called nilpotent if $A$ is associative and nilpotent:  $A^m = 0$ for some $m > 1$.
 
This research was inspired by discussions with Tim Kohl.  Many thanks to him for sharing his enthusiasm with me.

\section{Hopf Galois structures from nilpotent algebras}

Let $A = (A, +, \cdot)$ be a nilpotent $\Fp$-algebra of $\Fp$-dimension $n$.  The adjoint group $(A, \circ)$   of $A$ is the set $(A, \circ)$ with the circle operation $\circ$ defined by $a \circ b = a + b + a\cdot b$ for $a, b$ in $A$. The identity element is $0$.  Since $A$ is nilpotent, the circle inverse $a$ of $a$ is defined by 
\[ a = -a + a^2 - a^3 + \ldots ,\]
where $a^r = a \cdot a \cdot \ldots \cdot a$ ($r$ factors).  It is well known since [Ru07] that then $(A, \circ, +)$ is a left brace with additive group $(A, +)$, that is, for all $a, b, c$ in $A$
\[ a \circ(b + c) = a \circ b - a + a \circ c .\]
More recently ([Ch18]) it was observed that if $A^3 = \{ a\cdot b \cdot c | a, b, c \text{ in } A\}  = 0$, then $A$ is a bi-skew brace, that is, $(A, +, \circ)$ is also a skew brace with $(A, \circ)$ the additive group.  Thus

\begin{proposition}\label{key}  Let $(A, +)$ be an abelian $p$-group of finite order $p^n$, and let $(A, +, \cdot)$ be a nilpotent ring structure on $(A, +)$.  Let $(A, \circ)$ be the adjoint group on $A$ and let $\lambda_+, \lambda_{\circ}$ be the corresponding left regular representations of $A$ into $\Perm(A)$.  Then $\lambda_\circ(A)$ is normalized by $\lambda_+(A)$ if and only if $A^3 = 0$.  \end{proposition}

\begin{proof}    We have that for all $z$ in $A$, 
\[ \lb_{\circ}(x)(z) = x + z + x\cdot z \]
while
\[ \lb_{+}(y)(z) = y + z .\]

Suppose for each $x, y$ in $A$ there is some $w$ in $A$ so that
\[ \lb_+(y)\lambda_{\circ}(x)\lb_+(-y) = \lambda_{\circ}(w).\]
Applied to an element $z$ of $G$, the left side is 
\beal 
 \lb_+(y)\lambda_{\circ}(x)\lb_+(-y) (z) &= y + \lambda_{\circ}(x)(z-y)\\
&= y + x + (z-y) + x\cdot(z  - y)\\
&= x + z + x \cdot z - x \cdot y. \eeal
The right side is 
\[\lambda_{\circ}(w)(z) = w + z + wz .\]
So for all $x, y, z$ in $A$, we must have
\[ x + z + x\cdot z - x \cdot y = w + z + w\cdot z .\]
In particular, for $z = 0$, we have
\[ x - x\cdot y = w .\]
Then for all $z$ in $A$, we must have
\[ x \cdot z =  x\cdot z - x \cdot y \cdot z .\]
This holds if and only if $x \cdot y \cdot z = 0$.  
Thus if $A^3 = 0$, then for all $x, y$ in $A$, 
\[ \lambda_+(y)\lambda_{\circ}(x)\lambda_+(-y) = \lambda_{\circ}(x - x\cdot y) \]
so $\lambda_{\circ}(A)$ is normalized by $\lambda_+(A)$.  But if $A^3 \ne 0$, there exist $x, y$ in $A$ so that $ \lambda_+(y)\lambda_{\circ}(x)\lambda_+(-y) $ does not  $= \lambda_{\circ}(w)$ for any $w$ in $A$, and so $\lambda_{\circ}(A)$ is not normalized by $\lambda_+(A)$ in $\Perm(A)$.  \end{proof}  

We identify the corresponding  Hopf Galois structures:

\begin{corollary}  Let $L/K$ be a Galois extension with Galois group $(A, +) = G$, an abelian $p$-group of order $p^n$.  Let $A = (A, +, \cdot)$ be a commutative nilpotent ring structure on $(A, +)$ and suppose $A^3 = 0$.  Then $T = \lambda_{\circ}(A) \subset \Perm(A)$   yields a Hopf Galois structure  on $L/K$ by a $K$-Hopf algebra $H$, where 

i)  $H$ is the fixed ring of $LT$ under the action of $G$:
\[H = LT^G = \{ \sum_{x \in G} b_x\lambda_{\circ}(x) :  b_{x - x \cdot z} = b_x^z \text{ for all }z\text{ in }G \};\]
ii) $H$ acts on $L$ by 
\[ (\sum_{x \in G} b_x\lambda_{\circ})(a) = \sum_{x \in G} b_x a^{-x + x^2} \]
for $b, a$ in $L$.  \end{corollary}

\begin{proof}  Let $\{e_z : z \in G\}$ be the dual basis to the basis $G = (A, +)$ of the group ring $L[G]$.  
The action of $T = \lambda_{\circ}(A)$ on $GL = \sum_{z \in G} L e_z$ is by 
\[\lambda_{\circ}(x)(e_z) = e_{x \circ z}\]
 for $x$ in $G$.   This yields an action of the group ring $LT$ on $GL$ making $GL$ an $LT$-Hopf Galois extension of $L$.  Since $\lb_+(G)$ acts on $T$ by 
\[\lb_+(z)\lambda_{\circ}(x)\lb_+(-z) = \lambda_{\circ}(x - x\cdot z),\]
the corresponding $K$-Hopf algebra is 
\[  H = LT^G = \{ \sum_{x \in G} b_x \lambda_{\circ}(x):  b_{x - x \cdot z} = b_x^z \text{  for all } z \text{ in } G \}\]
where for $a$ in $L$ and $y$ in $G$, $a^y$ is the image of $a$ under the Galois action of $y$ on $L$, and $H$ acts on $GL$ by 
\[ ( \sum_{x \in G}b_x\lambda_{\circ}(x))( \sum_{y \in G}   a_ye_y) =   \sum_{x, y \in G} b_xa_y e_{x \circ y}. \]
Now $L$ embeds in $GL$ by
\[ a \mapsto  \sum_{y \in G} a^ye_y .\]
So the action of $H$ on $L$ is by 
\[ (\sum_{x \in G} b_x \lambda_{\circ}(x))(a)=  \sum_{x, y  \in G, x \circ y = 0} b_xa^y.\]
Since $x^3 = 0$, $x \circ (-x + x^2) = 0$,  and so the action of $H$ on $L$ can be written
\[ (\sum_{x \in G} b_x \lambda_{\circ}u(x))(a)=  \sum_{x \in G}  b_xa^{-x + x^2} .\]
\end{proof}

There is no a priori reason why $(A, \circ)$ and hence $T = \lambda_{\circ}(A)$ should be isomorphic to $G$, so that $H$ has type $G$.  But if $A$ is commutative, it is true:  we note the following variant of Theorem 1 of [FCC12]:

\begin{proposition}\label{FCC}  Let $p > 3$ be an odd prime and $G = (G, +)$ be a finite abelian $p$-group of order $p^n$.  Let $A = (G, +, \cdot)$ be a commutative nilpotent ring structure on $(G, +)$ and suppose $A^3 = 0$.  Then the regular subgroup  $N = (G, \circ)$ of $\Hol(G) \subset \Perm(G)$ is isomorphic to $(G, +)$.  \end{proposition}

The statement of Theorem 1 of [FCC12] replaces the condition $A^3 = 0$ in Proposition \ref{FCC} by the condition that the $p$-rank $m$ of $G$ should satisfy $m+1 < p$.   The proof of Proposition \ref{FCC} is essentially the same as that of Theorem 1 of [FCC12].  The only change is that the condition $A^3 = 0$ implies that $a^p = 0$ for all $a$ in $A$, which slightly simplifies the proof in [FCC12] by eliminating the need to apply a condition on the $p$-rank of $G$ to insure that $a^p$ does not interfere with the induction argument.

\section{Working in the affine group} 

 For the remainder of the paper we restrict $G$ to be an elementary abelian $p$-group of $p$-rank $n > 1$, and we consider only commutative nilpotent ring structures $A$ on $(G, +)$  with $A^3 = 0$.  From [Ch15], it is known that  the number of isomorphism types of such structures is bounded from below by  $p^b$ where $b = O(n^3)$.   

 Each such ring is an $\Fp$-algebra.  Let $\dim_{\Fp} A/A^2 = r$ and $\dim_{\Fp}A^2 = n-r$.  Given an $\Fp$-basis 
$( x_1, \ldots, x_{r})$ of $A/A^2$ and a basis $(y_1, \ldots y_{n-r})$ of $A^2$,  the multiplicative structure of $A$ with those bases is given by 
 a set $\Phi^{(k)} = (\phi_{i j}^{(k)})$ of $r \times r$ symmetric matrices with coefficients in $\Fp$, by the equations
 \[ x_ix_j = \sum_{k = 1}^{n-r} \phi_{i, j}^{(k)} y_k .\]
The group structure $(G, \circ)$ on $G$ arising from $(A, +, \cdot)$ depends on the matrices $\{ \Phi^{(k)}\}$, and hence so does the regular subgroup $T = \lambda_{\circ}(G)$ of $\Perm(G)$.   

It is convenient to view $\Hol(G)$ as the affine group $\Af_n(\Fp)$  and realize the regular subgroup $T$ inside $\Af_n(\Fp)$.

Let $\Af_n(\Fp)$ be the subset of $\GL_{n+1}(\Fp)$ consisting of matrices of the form
\[ \bpm B&v\\0&1 \epm ,\]
where $B$ is an $n \times n$ matrix, $v$ is a column vector in $\Fp^n$, $0$ is a $n$-row vector of zeros and 1 is a $1 \times 1$ identity matrix.  Then $\Af_n(\Fp)$ may be identified as the holomorph $\Hol(\Fp^n) = \lb(\Fp^n) \cdot \Aut(\Fp^n)$ of the additive group $\Fp^n$, where the matrices 
\[ \bpm P & 0\\0 & 1 \epm\]
with $P$ in $\GL_n(\Fp)$ form the subgroup $\Aut(\Fp^n)$ of $\Hol(\Fp^n)$,  and matrices 
\[ \bpm I &x\\0 & 1 \epm \]
for $x$ in $\Fp^n$ form the subgroup $\lambda(\Fp^n)$.

The group $T = \lambda_{\circ}(A)$ embeds  as a regular subgroup of $\Af_n(\Fp)$ as follows:

The map $\lambda_{\circ}$  from $\Fp^n$ to $\Fp^n$ is given by $\lambda_{\circ}(x)(y) = x\circ  y = x + y + x\cdot y$.  Write  $\lambda_{\circ}(x)(y) = x + y +  L_{x}(y)$, where $L_{x}(y) = x\cdot y$.  Then $L_{x}$ is a linear function from $\Fp^n$ to $\Fp^n$, so has a matrix relative to the standard basis of $\Fp^n$ that we also call $L_{x}$.   Then $\lambda_{\circ}(x)$ in $\Af_n(\Fp)$  becomes  the $n+1 \times n+1$ matrix
\[ T_{x} = \bpm I + L_{x} & x\\0 & 1\epm,\]
because for any $y$ in $\Fp^n$, we have
\[ T_{x} \bpm y\\1\epm = \bpm y + L_{x}(y)  +{x}\\1  \epm = \bpm y +{x}\cdot y + {x}\\1 \epm = \bpm \lambda_{\circ}(x)(y)\\1\epm.\]

\section{The case $r = n-1$}  

We now restrict to the class of examples where $\dim(A) = n,  \dim(A^2) = 1, A^3 = 0$.   Then  $A$ has the $\Fp$-basis $(z_1, \ldots, z_{n-1}, z_n)$ with $z_iz_j = \phi_{i, j}z_n$, so $A$ is determined by that basis and the single $n \times n$ structure matrix $\Phi = (\phi_{i  j})$.  Since $A^3 = 0$,  $\phi_{ni} = \phi_{in} = 0$ for all $i$.  In this section we determine the regular subgroups of $\Af_n(\Fp)$ associated to $A$.  

Since $A$ is commutative, the structure matrix $\Phi$ is symmetric.   Then  (c.f. [Wi09, Section 3.4.6)  there is an invertible matrix $P$ so that $P\Phi P^T = D = \diag(D_s, 0)$ is diagonal, where  $D_s = \diag(1, \ldots, 1, s)$ is $k \times k$ for some $k \le n$, where $s =1$ if $k$ is odd, and $s$ is either 1 or any fixed non-square in $\Fp$  when $k$ is even.

So set $z = Px$.  Then $z_n = x_n$ and with respect to the basis $(z_1,  \ldots, z_{n-1}, z_n)$,  $A$ has the structure matrix $D$ with $z_iz_j = d_{ij}z_n$ and 
\[  D = \diag(d_1, \ldots, d_n) = \diag(1, 1, \ldots, 1, s, 0, \ldots 0) \]
with $s =d_k $.

So by that change of basis of $A$, we can realize the group $T$ conveniently in 
\[ \Hol(G) \cong \Af_n(\Fp) = \bpm \GL_n(\Fp) & \Fp^n\\0 & 1 \epm.\]
by picking the basis $(z_1, \ldots, z_n)$ for $A$ so that $\Phi = D$.

 Let $\{e_1, \ldots, e_n \}$ be the standard basis of $\Fp^n$ corresponding to the basis $\{z_1, \ldots z_n\}$ of $A = (A, +)$.  Then $\lambda_\circ (z_i) = T_i $ is the element
\[ T_i = \bpm L_i & e_i\\0 & 1 \epm, \]
which acts on $A =\{r = \sum_{i=1}^n r_ie_i :  r \in \Fp^n\}$ embedded as elements $\binom {r}1 $ in $\Fp^{n+1}$ by
\beal  \bpm L_i & e_i\\0 & 1 \epm \bpm e_j\\1 \epm &= \bpm e_i \circ e_j\\1 \epm = \bpm e_i + e_j \\1\epm \text{  for } i \ne j \\
\bpm L_i & e_i\\0 & 1 \epm \bpm e_i\\1 \epm &= \bpm e_i \circ e_i\\1 \epm = \bpm e_i + e_i  + d_ie_n\\1 \epm.\eeal

\section{Determining the number of Hopf Galois structures associated to $A$}

In this section we determine the number of Hopf Galois structures on a Galois extension $L/K$ of fields with Galois group $G = (\Fp^n, +)$ that correspond to certain isomorphism types of nilpotent algebra structures on $G$.  

To do so, we have 
\begin{proposition}  Let $A$ be a nilpotent $\Fp$-algebra structure on $(\Fp^n, +)$.  Then the number of Hopf Galois structures on $L/K$ corresponding to the isomorphism type of $A$ is equal to 
\[ |\GL_n(\Fp)|/|\Sta(T)| \]
where $T = \lambda_\circ(A)$ is the regular subgroup of $\Af_n(\Fp)$ corresponding to $A$ and
\[ \Sta(T) = \{ P  \in \GL_n(\Fp) : \bpm P & 0\\0&1\epm  T = T \bpm P & 0\\0&1\epm.\]
\end{proposition}
This follows by [CDVS06], which showed that two nilpotent $\Fp$-algebras on $(\Fp^n, +)$ are isomorphic if and only if the corresponding regular subgroups of $\Af_n(\Fp)$ are conjugate by an element of $\Aut(G) = \bpm \GL_n(\Fp) & 0\\0 & 1 \epm$ in $\Af_n(\Fp)$.
 
 We note that given a regular subgroup $T$ of $\Af_n(\Fp)$ normalized by $\lb_+(A)$,  corresponding to $A$ with $A^3 = 0$, then all of the regular subgroups in the orbit of $T$ under conjugation by $\Aut(G)$ correspond to algebras $A_1$ isomorphic to $A$, hence have $A_1^3 = 0$.  Thus all are normalized by $\lb(G)$.   Hence by Galois descent, all of those regular subgroups give rise to Hopf Galois structures on a Galois extension $L/K$ with Galois group $G$.  
 
The commutative nilpotent $\Fp$-algebra $A$ with $A^2 = 0$  yields the classical Galois structure on a Galois extension with Galois group $G$.  For then  $\Phi = D = 0$, and the corresponding regular subgroup $T$ of $\Af_n(\Fp)$ is $\lb_+(A)$, which is stable under conjugation by every element of $\Aut(G)$, hence yields only the classical Galois structure on $L/K$.      

Since $p$ is odd, we may assume that  $\Phi  = \diag(D_s, 0)$. Let $\overline{v}_s = (r_1, r_2, \ldots, r_{k-1}, sr_k)^T$, $\overline{v} =  (r_1, r_2, \ldots, r_{k-1}, r_k)^T$,  and $\overline{w} = (r_{k+1}, \ldots, r_{n-1})^T$  (column vectors of elements of $\Fp$).  Then it is convenient to write elements of $T = \lambda_\circ(A)$ as  block matrices of the form 
\[ T = \{\lambda_{\circ}(r) =  \bpm I & 0 & 0 & \overline{v} \\ 0 & I & 0 &\overline{w} \\ \overline{v}_s ^T & 0 & 1 & r_n\\0 & 0 & 0 & 1 \epm :  r \in \Fp^n \} \]
where the diagonal entries are  identity matrices of size $k \times k$, $(n-1-k) \times (n-1-k)$, $1 \times 1$ and $1 \times 1$, respectively.

To determine the number of Hopf Galois structures corresponding to regular subgroups in the orbit of $T$, we need to find the stabilizer of  $T$ under conjugation by the elements of $\Aut(G) = \GL_n(\Fp)$.  

To determine the stabilizer of $T$, we   seek the set of $(n+1) \times (n+1)$ matrices 
\[  Q = \bpm P & 0\\0 & 1 \epm \]
in $\Aut(G) \subset \Af_n(G)$ so that $QTQ^{-1} = T$, where $P$ in $\GL_n(\Fp)$ has the form
\[ P = \bpm  P_{11} & P_{12} & P_{13}\\P_{21} & P_{22} & P_{23}\\ P_{31} & P_{32} & P_{33} \epm \]
with blocks of the same size as $\lambda_{\circ}(r)$.  Let $\lambda_{\circ}(r')$ be another element of $T$. 
We compute $P\lambda_{\circ}(r)$  and $\lambda_{\circ}(r')P$  and set $P\lambda_{\circ}(r) =\lambda_{\circ}(r')P$.

Equating the (11) terms yields that $P_{13} = 0$.  

Equating the (21) terms yields that $P_{23} = 0$.  

Equating the (32) terms yields that $P_{12} = 0$.  

Then equating the (31) terms yields
\[\overline{v}_s'^TP_{11} = P_{33}\overline{v_s}^T.\]

Equating the (14) terms yields
\[ P_{11}v = v' .\]

Equating the (24) terms yields
\[ w' = P_{21}v + P_{22}w.\]

Equating the (34) terms yields 
\[ t' = P_{31}v + P_{32} w + P_{33} t .\]

The (24) and (34) equations define $w'$ and $t'$.  Setting $P_{33} = q$, a non-zero element of $\Fp$, then from (14) and (31) we have
\[ P_{11}^T v_s' = q  v_s  \text{  and  }  P_{11}v =v' .\]
Recalling that $D_s = \diag(1, \ldots 1, s)$, a $k \times k$ matrix, then $D_s v = v_s$, $D_s v' = v_s'$.  So
\[ P_{11}^TD_sv' = qD_sv, \]
hence
\[ P_{11}^TD_sP_{11}v =  qD_s v.\]
Thus $P$ is in the stabilizer of $T$ if 
\[ P =  \bpm  P_{11} &0 & 0 &0 \\P_{21} & P_{22} & 0&0\\ P_{31} & P_{32} & P_{33}&0\\0&0&0&1 \epm \]
where 

$P_{33} = q$ is in $\GL_1(\Fp)$;

$P_{32}$ is $1 \times (n-1-k)$ and arbitrary;

$P_{31}$ is $1 \times k$ and arbitrary;

$P_{22}$ is in $\GL_{n-1-k}(\Fp)$;

$P_{21}$ is $(n-1-k) \times k$ and is arbitrary; and

$P_{1 1}$ is in $\GL_k(\Fp)$ and satisfies   $P_{11}^T D_s P_{11} = qD_s$.  

As noted above,  we may assume that $A$ has a basis for which $A$ has the structure matrix 
  $D =  \bpm D_s &0\\0 & 0 \epm$ where $D_s = \diag(1, 1,\ldots, s)$  is  $k \times k$,  $k < n$ and
  
1) $k$ is odd and $s = 1$;

2) $k$ is even  and $s = 1$;  

3)  $k$ is even and $s$ is a non-square in $\Fp$. 

We may then determine the possible $P_{11}$ in each of the three cases.  The notation for the orthogonal groups over $\Fp$ is from [Wi09],  Section 3.7.

\begin{proposition} \label{sta}
For Case 1), let $k = 2m+1$ and $s = 1$.  For all $q \ne 0$ in $\Fp$, there exists a $k \times k$ matrix $C$ so that $C^TC = qI$ if and only if $q$ is a square.  Fixing  $C$, then  $P_{11}^TP_{11} = qI$ if and only if $P_{11} = CU$ for $U$ in $\GO_{2m+1}(\Fp)$.

For Case 2), let $k = 2m$ and $s = 1$.  For all $q \ne 0$ in $\Fp$, there exists a $k \times k$ matrix $C$ so that $C^TC = qI$.  Fixing  $C$, then $P_{11}^TP_{11} = qI$ if and only if $P_{11} = CU$ for $U$ in $\GO_{2m}^+(\Fp)$.

For Case 3), let $k = 2m$ and $s$ be a non-square in $\Fp$.   For all $q \ne 0$ in $\Fp$, there exists a $k \times k$ matrix $C$ so that $C^TD_s C = qD_s$. Fixing  $C$, then $P_{11}^TD_sP_{11} = qD_s$ if and only if $P_{11} = CU$ for $U$ in $\GO_{2m}^-(\Fp)$.
\end{proposition} 

\begin{proof}  In Case 1) with $k$ odd,  if there exists $C$ so that $C^TC = qI$, then taking determinants gives $\det(C)^2 = q^k$, hence $q$ must be a square.  

For the rest, it suffices to find the matrix $C$ in each case.  

For Case 1), let  $q = t^2$, then $C = tI$ satisfies $C^TC = qI$.

For Case 2), let $q = f^2 + g^2$, let $Q = \bpm f&g\\-g & f\epm$ and let $C = \diag(Q, Q, \ldots, Q)$.  Then $C^TC = qI$. 

For Case 3), let $q = f^2 + g^2$ and $Q$ as in Case 2). For $s$ a non-square in  $\Fp$,  find $w$ and $x$ in $\Fp$ so that $w^2 + sx^2 = q$.  (If $q$ is a square, let $x = 0, w^2 = q$; if $q$ is a non-square, let $w = 0$ and find $x$ so that $sx^2 = w$, possible because the squares have index 2 in $\Fp^{\times}$.)  Then $R = \bpm w & sx\\x & -w \epm$ satisfies $R^T\bpm  1&0\\0&s \epm R =  \bpm q & 0\\0 & sq \epm$.   Let  $C = \diag(Q, Q, \ldots, Q, R)$.  Then $C^TD_sC = qD_s$.  
\end{proof}

\begin{corollary}  Let $A$ be a commutative nilpotent $\Fp$-algebra of dimension $n$ with $A^3 = 0$ and $\dim(A^2) = 1$.  Suppose the structure matrix of $A$ is $\Phi = \diag(D_s, 0)$ where $D_s$ is $k \times k$ and

1)  $k = 2m+1, s = 1$

2)  $k = 2m, s = 1$

3  $k = 2m, s $ is a non-square in $\Fp$.  

Then the number of distinct regular subgroups  of $\Af_n(\Fp)$ associated to $A$, and hence the number of Hopf Galois structures on $L/K$ associated to the isomorphism type of $A$,  is

1)  
\[ \frac {|\GL_n(\Fp)|}{ (\frac {p-1}2 )\cdot |\GO_{2m+1}| \cdot |GL_{n-1-k}| \cdot p^{k(n-1-k) + (n-1)} }\]

2) 
\[ \frac {|\GL_n(\Fp)|}{ (p-1 )\cdot |\GO^+_{2m}| \cdot |GL_{n-1-k}| \cdot p^{k(n-1-k) + (n-1)}} \]

3)
 \[ \frac {|\GL_n(\Fp)|}{ (p-1 )\cdot |\GO^-_{2m}| \cdot |GL_{n-1-k}| \cdot p^{k(n-1-k) + (n-1)}} \]

\end{corollary}

The orders of the $k \times k$ orthogonal groups are polynomials in $p$ of degree  $(k^2 -k)/2$ (c.f. [Wi09], p. 72), and the order of $\GL_n(\Fp)$ is a polynomial of degree $n^2$.  Hence we have 

\begin{corollary} \label{degree} Let $A$ be a commutative nilpotent $\Fp$-algebra of dimension $n$ with $A^3 = 0$, $\dim(A^2) = 1$ and structure matrix of rank $k$.  Let $L/K$ be a Galois extension with Galois group $G \cong (A, +)$.  Then the number of Hopf Galois structures on $L/K$ of type $(A, \circ)$  is a polynomial function of $p$ of degree 
\beal & n^2 - (k^2 - k)/2 - (n-k)(n-1) - 1\\
& = \frac {(2n -k)(k+1)}2 -1.\eeal
\end{corollary}
The number of Hopf Galois structures increases with $k$ and is maximal when $k = n-1$.  
\section{ For $n = 2, 3, 4$}

We compare the counts of Hopf Galois structures  in the last section to the number of Hopf Galois structures found by formal group methods in [Ch05] for $n = 2, 3$.

\subsection*{The case $n=2$}
Let $n = 2,  k = 1$.  Then  $\Phi = (1)$.  For $P$ to stabilize $T$,  
\[ P = \bpm P_{11} & 0\\P_{21} & P_{22} \epm ,\]
and the number of choices for each submatrix in $P$ is
\[ \bpm |\GO_1| & 1 \\ p & \frac {p-1}2 \epm. \]
Since $\GO_1 = \{(1), (-1)\}$,  the size of the stabilizer of the regular subgroup is 
\[ 2 \cdot  p \cdot \frac {p-1}2  = p(p-1).\]
 The order of $GL_2(\Fp)$ is $(p^2-1)(p^2 -p)$.  So there are $p^2 -1$ distinct regular subgroups in the orbit of the regular subgroup corresponding to $\Phi$.

Since every nilpotent algebra structure $A$ on $(\Fp^2, +)$ has $A^3 = 0$, we have counted all Hopf Galois structures on a Galois extension with Galois group $C_p^2$.

\subsection*{The case  $n = 3$} 

\

Subcase:  $k = 1$:   The matrix $P$ is in the stabilizer of the regular subgroup $T$ corresponding to $\Phi = \diag(1, 0)$  if 
\[ P = \bpm P_{11} & 0 & 0\\P_{21} & P_{22} & 0\\P_{31} & P_{32} & P_{33} \epm, \]
all submatrices being $1 \times 1$.  So the number of choices for each entry is
\[  \bpm  |\GO_1| & 1 & 1\\ p & |\GL_1| & 1\\p & p & \frac {p-1}2 \epm.\]
Then $|\GO_1| =2$ and $|\GL_1| = p-1$, so the size of the stabilizer $\Sta(T)$ is 
\[ p^3(p-1)^2 ,\]
and the orbit has cardinality  
\[ |\GL_3(\Fp)|/|\Sta(T)| = (p-1)^3(p+1).\]

\

Subcase:  $k= 2$, $s = 1$, $\Phi = \diag(1, 1)$:  The matrix $P$ is in the stabilizer if 
\[ P = \bpm P_{11} & 0 \\P_{21} & P_{22} \epm, \]
where $P_{11} $ is in $\GO_2^+$.  The number of choices for each submatrix is 
\[  \bpm |\GO_2^+| & 1 \\p^2 & p-1 \epm, \]
and $|\GO_2^+| = 2(p-1)$, so the size of the stabilizer is
\[ 2(p-1)^2 p^2.\]

\

Subcase:  $k= 2$, $s$ a non-square, $\Phi = \diag(1, s)$:  The matrix $P$ is in the stabilizer if 
\[ P = \bpm P_{11} & 0 \\P_{21} & P_{22} \epm, \]
where $P_{11} $ is in $\GO_2^-$.  The number of choices for each submatrix is 
\[  \bpm |\GO_2^-| & 1 \\p^2 & p-1 \epm, \]
and $|\GO_2^-| = 2(p+1)$, so the size of the stabilizer is
\[ 2(p^2-1) p^2.\]

The number of regular subgroups corresponding to each case is $|\GL_3|$ divided by the order of the stabilizer:

For $k=1$, the number of regular subgroups is \[(p^3-1)(p+1).\]

For $k=2, s = 1$, the number of regular subgroups is \[(p^3-1)p(p+1)/2.\]

For $k=2, s$ a non-square,  the number of regular subgroups is 
\[(p^3-1)p(p-1)/2.\]

These agree with the counts found in [Ch05].

The only isomorphism type of nilpotent algebras $A = (\Fp^3, +, \cdot)$ with $A^3 \ne 0$ is the algebra with $\dim(A/A^2) = 1$.  

\subsection*{The case $n= 4$} 
This case has not previously been looked at.  

For $n = 4$ there are four subcases:

\

$k=1$.  Here 
\[ P = \bpm P_{11} & 0 & 0\\P_{21} & P_{22} & 0\\P_{31} & P_{32} & P_{33} \epm, \]
 where $P_{22}$ is $2 \times 2$.  The number of choices for each submatrix is
\[ \bpm |\GO_1| & 1 & 1\\ p^2 & |\GL_2| & 1\\p & p^2 & \frac {p-1}2 \epm.\]
So the size of the stabilizer is
\[ 2p^5  (p^2-1)(p^2-p)(\frac {p-1}2) .\]

\

$k = 2, s =1$:  Here $P_{11}$ is $2 \times 2$.  The number of choices for each matrix is
\[ \bpm |\GO_2^+| & 1 & 1\\ p^2 & |\GL_1| & 1\\p^2 & p & p-1 \epm.\]
So the size of the stabilizer is
\[ 2p^5 (p-1)^3 .\]

\

$k=2$, $s$ a non-square.  It is the same as the last case except $P_{11}$ is in $\GO_2^-$, so the size of the stabilizer is 
\[ 2p^5(p-1)^2(p+1).\]

\

$k= 3$.  Here 
\[ P = \bpm P_{11} & 0\\P_{21} & P_{22} \epm\]
where $P_{11}$ is in $\GO_3$, which has order $2p(p^2-1)$, and $P_{22} = (q)$ where $q$ is a square.  So the order of the stabilizer is
\[ 2p(p^2-1)p^3 \frac{(p-1)}2.\]

The number of regular subgroups in each case is the order of $\GL_4(\Fp)$ divided by the orders of the respective stabilizers:

\begin{center}  \begin{tabular} 
{c|c}
Case & number of regular subgroups\\ \hline
$k=1$ & $(p^2+1)(p+1)(p^3-1)$\\
$k = 2, s=1$ & $p(p^2+1)(p^3-1)(p+1)^2/2$\\
$k = 2, s$ a non-square & $p(p^4-1)(p^3-1)/2$\\
$k = 3$ & $p^2(p^4-1)(p^3-1)$\\ \hline
\end{tabular} \end{center}
 
 The total number of Hopf Galois structure exceeds $p^9$.  Note that the degrees of the polynomials in each case agree with Corollary \ref{degree}.

\subsection*{The Hopf Galois structure}  Given a Galois extension $L/K$ with Galois group $G \cong \Fp^n$, if the commutative nilpotent algebra $A$ with $\dim(A^2) = 1, A^3 = 0$ has diagonal structure matrix  $\Phi = \diag(d_1, \ldots, d_k, 0, \ldots, )$, then the regular subgroup $T$ corresponding to $D$ acts on $GL$ by 
\[ \lambda_{\circ}(r)(e_{t}) = e_{r \circ t} = e_{w}\]
where 
\[ w = r + t + (\sum_{i=1}^k r_it_id_i)x_n ,\]
and $\lb(G)$ conjugates $T$ by 
\beal \lb(t)\lambda_{\circ}(r)\lb(-t)  &= \lambda_{\circ}(r - r\cdot t)\\
&= \lambda_{\circ}(r - \sum_{i=1}^k r_it_id_i)x_n). \eeal

\section{The Galois correspondence ratio}

 Let  $A$ be a commutative $\Fp$-algebra of dimension $n$ with $A^3 = 0$, yielding the skew brace $(A, +, \circ)$.  In [Ch17] (generalized in [Ch18] and [Ch19]) we showed that for a Galois extension $L/K$ with Galois group $(A, +)$ and an $H$-Hopf Galois structure of type $(A, \circ)$,  the image of the Galois correspondence for the Hopf Galois structure is in bijective correspondence with the ideals of $A$.  Thus the Galois correspondence ratio for the Hopf Galois structure is 
 \[ GC(L/K, (A, +), H) = \frac{|\{\text{ ideals of $A$}\}|}{|\{\text{ subgroups of $(A, +)$}\}|}.\]
 We have
 \begin{proposition}
 Let $A$ be as in Corollary 5.3 with a non-zero structure matrix $D_s$ of rank $k \ge 1$.  Let $L/K$ be a Galois extension with Galois group $G \cong (A, +)$ and a Hopf Galois structure associated to the skew brace $(A, +, \circ)$.  Then 
\beal  GC(L/K, (A, +), H) &= O(\frac{1}{p^{(n-1)/2}})\text{ for $n$ odd }; \\&= O(\frac{1}{p^{n/2}})\text{ for $n$ even}. \eeal
 \end{proposition}
 \begin{proof}
 The denominator of $GC(L/K, (A, +), H)$ is equal to the number of subspaces of $\Fp^n$, a known quantity.   So to estimate this ratio, we need to estimate the number of ideals of $A$.
 
 The algebra $A = (A, n, k)$ has basis $(x_1, x_2, \ldots, x_n)$ where $x_i^2 = x_n$ for $i = 1, 2, \ldots, k-1$, $x_k^2 = s \ne 0$ and $x_i^2 = 0$ for $k < i \le n$.  Viewing $A$ as a vector space with basis $x_1, \ldots x_n$, we know (c.f [CG18], Section 1) that the number of subspaces of $\Fp^n$ of dimension $k$ is $\ge p^{k(n-k)}$ and has order of magnitude $= p^{k(n-k)}$.  

We can count the total number of subspaces of $A$ by viewing the subspaces of $A$ as row spaces of $n \times n$ matrices with entries in $\Fp$ and counting the number of parameters of all possible reduced row echelon forms of those $n \times n$ matrices.  So let $R = (c_1, c_2, \ldots, c_r)$ denote the general reduced row echelon form of rank $r$ with $r$ non-zero rows and pivots in columns numbered $c_1, c_2, \ldots, c_r$.  Then the number $n_R$ of $\Fp$-parameters in the matrix $R = (c_1, c_2, \ldots, c_r)$, (counting row by row from the top) is equal to 
 \[(c_2 - c_1-1)+ 2(c_3 - c_2 -1)+ \ldots  + \ldots+ (r-1)(c_r - c_{r-1} -1) +  r(n+1 - c_r -1).\]
 So the dimension of the subspace defined by the matrix $R$ is 
\beal m_R &= p^{n_R}\\& = p^{c_2 - c_1-1} \cdot p^{2(c_3 - c_2 -1)}  \cdot \cdots  \cdot p^{(r-1)(c_r - c_{r-1} -1)}  \cdot p^{r(n+1 - c_r -1)}.\eeal
The largest $m_R$ can be is if $(c_1, c_2, \ldots, c_r) = (1, 2, \ldots, r)$, so that the product reduces to the single term  $p_r = p^{r(n+1- r-1)}$.  Thus for $n$ even, the number $s(\Fp^n)$of subspaces  of $\Fp^n$ is a polynomial in $p$ with a unique highest degree term, when $r = n/2$, namely   $p^{n^2/4}$.  For $n$ odd, $s(\Fp^n)$ is a polynomial in $p$ with two equal highest degree terms, when $r = (n-1)/2$ or $r = (n+1)/2$, namely $p^{(n^2 -1)/4}$.  Thus the leading term of $s(\Fp^n)$ for $n$ odd is $= 2p^{(n^2 -1)/4}$.

Now we estimate the number of ideals of $A$, assuming that in $A$, $x_1^2 = dx_n$ with $d \ne 0$.  The key fact is that if a matrix $R$ represents a subspace which is an ideal and contains an element $x =x_1 +a_2 x_2 + \ldots + a_nx^n$, then it also contains $x_1x = dx_n$.  So $R$ must contain a row $(0, 0, \ldots, 0, 1)$.  Thus the matrices 
\[ R = (1, 2, 3, \ldots, r)\] which give the largest number of parameters do not represent ideals, while the matrices
\[R_I = (1, 2, 3, \ldots, r, n)\] do represent ideals, but have $r$ fewer parameters than $R$ does.  Also
$R' = (2, 3, \ldots r)$ represents an ideal if $x_2^2 = 0$, but $R'$ has $n-r$ fewer parameters than $R$.   In particular, for $n$ odd, the matrix  $R_I$ giving the most parameters is 
\[ R_I = (1, 2, 3, \ldots, \frac{n-1}2, n), \]
namely, $(n^2-1)/4$ parameters, and for $n$ even, the matrix  $R_I$ giving the most parameters is 
\[ R_I = (1, 2, 3, \ldots, \frac{n-1}2, n), \]
namely $\frac{n^2}4$.

Thus in the case of an $A$ closest to the trivial algebra $A = (\Fp^n, +)$, the ratio
\[ \#\{\text{ideals of $A$}\}/\#\{\text{subspaces of $A$}\} = O(1/(p^{(n-1)/2}))\text{ or } O(1/p^{n/2}) \]
for $n$ odd, resp. even.  \end{proof}

If $A$ has $x_i^2 = s_ix_n$ for $s_i \ne 0$ for $i = 1, \ldots, d$, the number of subspaces that are ideals decreases as $d$ increases, to the point where if $d$ = $n-1$, then the ideals of $A$ are the subspaces of $A$ that contain $x_n$.  Then the number of non-zero ideals of $A$ is equal to the number of subspaces of $\Fp^{n-1}$.

\

\end{document}